\documentclass[journal]{IEEEtran}

\usepackage{amsmath,amssymb,amsfonts}
\usepackage[ruled]{algorithm2e}
\usepackage{graphicx}
\usepackage{textcomp}
\usepackage{xcolor}
\usepackage{verbatim}
\usepackage{float}
\floatstyle{plaintop}
\restylefloat{table}
\usepackage{subfigure}
\usepackage{amssymb}
\usepackage[ruled]{algorithm2e}
\usepackage{algorithmic}
\usepackage{booktabs}
\usepackage{amsthm}
\usepackage{amsmath}

\usepackage[
  open,
  openlevel=2,
  atend,
  numbered
]{bookmark}

\newcommand{\norm}[1]{\left\lVert#1\right\rVert}

\newtheorem{theorem}{Theorem}
\newtheorem{lemma}{Lemma}

\begin{document}

\title{Low-Complexity Iterative Methods for Complex-Variable Matrix Optimization Problems \\in Frobenius Norm}

\author{ Sai Wang, \IEEEmembership{Student Member, IEEE}, and Yi Gong, \IEEEmembership{Senior Member, IEEE}

\thanks{The authors are with the department of electrical and electronic engineering, Southern University of Science and Technology, Shenzhen, China (e-mail: gongy@sustech.edu.cn).}}

\maketitle

\begin{abstract}
Complex-variable matrix optimization problems (CMOPs) in Frobenius norm emerge in many areas of applied mathematics and engineering applications. In this letter, we focus on solving CMOPs by iterative methods. For unconstrained CMOPs, we prove that the gradient descent (GD) method is feasible in the complex domain. Further, in view of reducing the computation complexity, constrained CMOPs are solved by a projection gradient descent (PGD) method. The theoretical analysis shows that the PGD method maintains a good convergence in the complex domain. Experiment results well support the theoretical analysis.
\end{abstract}

\begin{IEEEkeywords}
Complex variables, Matrix optimization, Convergence Analysis.
\end{IEEEkeywords}

\IEEEpeerreviewmaketitle

\section{Introduction}

\IEEEPARstart{T}{he} real-valued problems in complex variables have widely arisen in control theory \cite{control}, medical imaging \cite{medicalimage}, and signal processing \cite{signal}, especially in waveform design \cite{beam}. Unfortunately, they are necessarily nonanalytic. A traditional method of solving complex-variable matrix optimization problems (CMOPs) uses the derivatives with respect to the real and imaginary parts separately to generate a descent direction. Then the combination or optimization is performed in an augmented space by converting the complex domain into the real domain of double the dimension. By reformulating an optimization problem that is inherently complex to the real domain, it is easy to miss some important information about the physical characteristics of the original problem. Moreover, unnecessarily long expressions will suffer from high computational complexity and a slow convergence rate when dealing with a large-scale problem. Hence, developing the complex-variable optimization method with low computational complexity is needed. To overcome this challenge, the calculus underlying complex derivatives was developed by Wirtinger in the early 20th century. Recent complex optimization theory has shown that the Wirtinger calculus allows a mathematically robust definition of a gradient operator in the complex domain \cite{LaurentSorber2012}, which can be done without considering the real and imaginary parts separately and without doubling the dimension. In other words, the complex optimization methods are performed directly in the complex domain. To cope with the complex-valued nonlinear programming problem with linear equality constraints, the authors in \cite{SongchuanZhang2018,7353156} proposed a complex-valued optimization method that has a global convergence under mild conditions. In \cite{SitianQin2018}, a one-layer recurrent neural network was proposed for solving the constrained real-valued problem in complex variables. The state of the proposed neural network has a lower model complexity and better convergence. Note that the above complex-variable problems are limited to the complex vector. The authors in \cite{SongchuanZhang2021} proposed a complex-valued projection neural network in a matrix state space to solve a complex-variable basis pursuit problem. In addition, the complex-value step size was investigated for the complex gradient method  \cite{7321072}. To our knowledge, there is little systematic analysis of convex CMOPs in Frobenius norm.

\textbf{Notation and preliminaries.}
Vectors are denoted by boldface lowercase letters, e.g., $\boldsymbol{a}$. Matrices are denoted by boldface capital letters, e.g., $\boldsymbol{A}$. The $i$-th entry of a vector $\boldsymbol{a}$ is denoted by $a_i$, element $(i,j)$ of a matrix $\boldsymbol{A}$ by $a_{i,j}$. The superscripts $(\cdot)^T$, $(\cdot)^H$ and $(\cdot)^{-1}$ are used for the transpose, Hermitian transpose, and matrix inverse, respectively. The real part of Frobenius inner product and Frobenius norm are denoted by $\mbox{Re}[\cdot,\cdot]_F$ and $\norm{\cdot}_F$, respectively.
In order to facilitate the division of the real number and pure imaginary number, for a complex number $a\in \mathbb{C}$, the real part of $a$ is $\mbox{Re}(a)=\bar{a}$ and the imaginary part of $a$ is $\mbox{Im}(a)=\ddot{a}$, respectively.
Thus a complex matrix is written as $\boldsymbol{W}=\bar{\boldsymbol{W}}+\ddot{\boldsymbol{W}}i$.

\section{Unconstrained CMOPs in Frobenius Norm}\label{section2}
In this section, we derive the optimal condition of the unconstrained CMOP. Then we prove that the gradient descent (GD) method is feasible in the complex domain. An unconstrained CMOP in Frobenius norm $\mathcal{F}(\boldsymbol{W}): \mathbb{C}^{N\times K}\rightarrow \mathbb{R}$ is written as
\begin{equation}
\begin{aligned}
\mathcal{P}1\quad \min_{\boldsymbol{W}}  \mathcal{F}(\boldsymbol{W})= \frac{1}{2}\norm{\boldsymbol{H}\boldsymbol{W}-\boldsymbol{A}}_{F}^2.
\end{aligned}
\end{equation}
where $ \boldsymbol{H}\in \mathbb{C}^{M\times N},\boldsymbol{W}\in \mathbb{C}^{N\times K}, \boldsymbol{A}\in \mathbb{C}^{M\times K}$. Provide $N\leq M$ condition.

\begin{lemma}
For an unconstrained CMOP $\mathcal{F}(\boldsymbol{W}): \mathbb{C}^{N\times K}\rightarrow \mathbb{R}$, $\mathcal{F}(\boldsymbol{W})$ is convex for the real and imaginary variables.\label{lemma1}\end{lemma}
\begin{proof}
It is easy to prove by the definition of convexity so the proof is omitted.
\end{proof}

\begin{theorem}
For an unconstrained CMOP, if $\boldsymbol{H}^H\boldsymbol{H}$ is invertible, there exists a unique optimal solution $\boldsymbol{W}^{opt}=(\boldsymbol{H}^H\boldsymbol{H})^{-1}\boldsymbol{H}^H\boldsymbol{A}$.\label{the1}
\end{theorem}
\begin{proof}
Please refer to Appendix A.
\end{proof}
Theorem \ref{the1} shows that the optimal condition in the real domain could be directly applied to the unconstrained CMOP.

\subsection{GD Method and Convergence Analysis}
In this part, we investigate the iterative method to solve $\mathcal{P}1$ in view of avoiding matrix inversion. Most importantly, for linear systems, the iterative method has lower complexity than the closed-form solution \cite{sda}. To achieve this goal, we derive the first-order condition of convexity. Based on the Lipschitz condition, we analyze the convergence of the GD method in the complex domain. For the $t$-th iteration, the iterative formula of the GD method satisfies:
\begin{equation}
\boldsymbol{W}^{t+1}=\boldsymbol{W}^{t}-\alpha\nabla\mathcal{F}(\boldsymbol{W}^{t}), \label{3.2}
\end{equation}
where $\alpha$ is the step size and $\nabla\mathcal{F}(\boldsymbol{W})=\boldsymbol{H}^H\boldsymbol{H}\boldsymbol{W}-\boldsymbol{H}^H\boldsymbol{A}$.
\begin{lemma}\label{lemma2}
For an unconstrained CMOP $\mathcal{F}(\boldsymbol{W})$, the first-order condition of convexity satisfies
\begin{equation}
\begin{aligned}
\mathcal{F}(\boldsymbol{W}^{})\ge \mathcal{F}(\boldsymbol{W}^{t})+\mbox{Re}[\boldsymbol{W}^{}-\boldsymbol{W}^{t},\nabla\mathcal{F}(\boldsymbol{W}^{t})]_F.
\end{aligned}\label{us3}
\end{equation}
\end{lemma}
\begin{proof}
Please refer to Appendix B.
\end{proof}

\begin{lemma}\label{lemma3}
The unconstrained CMOP satisfies a Lipschitz condition if a constant $L>0$ exists with
\begin{equation}
\begin{aligned}
\mathcal{F}(\boldsymbol{W}^{t+1})\leq \mathcal{F}(\boldsymbol{W}^{t})+\frac{L}{2}\norm{\boldsymbol{W}^{t+1}-\boldsymbol{W}^{t}}_F^2\\
+\mbox{Re}[\nabla\mathcal{F}(\boldsymbol{W}^{t}),\boldsymbol{W}^{t+1}-\boldsymbol{W}^{t}]_F.\label{2.15}
\end{aligned}
\end{equation}
\end{lemma}
\begin{proof}
Please refer to Appendix C.
\end{proof}
\begin{theorem}
If the step size satisfies $\alpha\in (0,{2}/{L})$, the GD method converges to the optimal solution of the problem $\mathcal{P}1$ in the complex domain. And it satisfies 
\begin{equation}
    \mathcal{F}(\boldsymbol{W}^{t})-\mathcal{F}(\boldsymbol{W}^{t+1})\ge \alpha(1-\frac{\alpha L}{2})\norm{\nabla\mathcal{F}(\boldsymbol{W}^{t})}_F^2.
\end{equation}
\end{theorem}
\begin{proof}
By lemma \ref{lemma3}, we have.
\begin{equation}
\begin{aligned}
\mathcal{F}(\boldsymbol{W}^{t+1})&\leq \mathcal{F}(\boldsymbol{W}^{t})+\mbox{Re}[\nabla\mathcal{F}(\boldsymbol{W}^{t}),\boldsymbol{W}^{t+1}-\boldsymbol{W}^{t}]_F\\
&\qquad \qquad \qquad \qquad   +\frac{L}{2}\norm{\boldsymbol{W}^{t+1}-\boldsymbol{W}^{t}}_F^2\\
&\leq \mathcal{F}(\boldsymbol{W}^{t})+\mbox{Re}[\nabla\mathcal{F}(\boldsymbol{W}^{t}),-\alpha\nabla\mathcal{F}(\boldsymbol{W}^{t})]_F\\
&\qquad \qquad \qquad \qquad  +\alpha^2\frac{L}{2}\norm{\nabla\mathcal{F}(\boldsymbol{W}^{t})}_F^2\\
&\leq \mathcal{F}(\boldsymbol{W}^{t})-\alpha(1-\frac{\alpha L}{2})\norm{\nabla\mathcal{F}(\boldsymbol{W}^{t})}_F^2.
\end{aligned}
\end{equation}
Thus this theorem holds.
\end{proof}
Since the projection gradient descent (PGD) degenerates into GD when the domain is full space, the sequence convergence of GD is similar to that of PGD. For detials, please refer to Theorem 3.

\section{Constrained CMOPs in Frobenius Norm}\label{section3}
In this section, a common nonlinear constrained CMOP in Frobenius norm is investigated. To solve this issue, the PGD method is adopted. Then we provide the convergence analysis of PGD.
\begin{equation}
\begin{aligned}
\mathcal{P}2 \quad\min_{\boldsymbol{W}}  \mathcal{F}(\boldsymbol{W})=&\quad \frac{1}{2}\norm{\boldsymbol{H}\boldsymbol{W}-\boldsymbol{A}}_{F}^2\\
\quad s.t. & \quad\mbox{diag}(\boldsymbol{W}\boldsymbol{W}^{H})\leq \boldsymbol{1}\eta,
\end{aligned}
\end{equation}
where $\mbox{diag}(\cdot)$ is a matrix-to-vector operator that denotes the vector formed by the diagonal elements of matrices and $\boldsymbol{1}=[1,\cdots,1]^{T}$ stands for  the all-one vector. 
$\mathcal{P}2$ is a quadratically constrained quadratic programming. The Lagrangian multiplier is
\begin{equation}
\begin{aligned}
\mathcal{L}( \boldsymbol{W}, \boldsymbol{\lambda})=&\norm{\boldsymbol{H}\boldsymbol{W}-\boldsymbol{A}}_{F}^2+\boldsymbol{\lambda}^T (\mbox{diag}(\boldsymbol{W}\boldsymbol{W}^H)-\boldsymbol{1}\eta).\\
\end{aligned}
\end{equation}
By the Karush-Kuhn-Tucker (KKT) condition, we have 
\begin{equation*}
\begin{aligned}
\nabla\mathcal{L}(\boldsymbol{W}^{*},\boldsymbol{\lambda}^{*})=
[\boldsymbol{H}^H\boldsymbol{H}+\mbox{Diag}(\boldsymbol{\lambda}^*)]\boldsymbol{W}^{*}-\boldsymbol{H}^H\boldsymbol{A}=0\\
(\boldsymbol{\lambda}^*)^T (\mbox{diag}(\boldsymbol{W}^*(\boldsymbol{W}^*)^H)-\boldsymbol{1}\eta)=0,\\
(\mbox{diag}(\boldsymbol{W}^*(\boldsymbol{W}^*)^H)- \boldsymbol{1}\eta)\leq 0,\\
\boldsymbol{\lambda}^*\ge 0,
\end{aligned}
\end{equation*}
where $\mbox{Diag}(\cdot)$ is a vector-to-matrix operator. The optimal solution is equal to $\boldsymbol{W}^{*}=[\boldsymbol{H}^H\boldsymbol{H}+\mbox{Diag}(\boldsymbol{\lambda}^*)]^{-1}\boldsymbol{H}^H\boldsymbol{A}.$
The worst-case complexity to obtain the closed-form solution is $\mathcal{O}(2^N(N^2M+NMK+N^3))$. It is unacceptable for large dimensions of $\boldsymbol{W}$. To reduce the complexity, the PGD method \cite{BingshengHe2002} is adopted with $\mathcal{O}(N^2M+NMK+tN^2K)$.

\subsection{PGD Method and Convergence Analysis}
In this part, the PGD method is implemented to solve the constrained CMOP in Frobenius norm. Then we analyze the convergence of the PGD method in the complex domain. 
 The iterative formula of the PGD method is written as
\begin{equation}
\boldsymbol{W}^{t+1}=P_{\Omega}(\boldsymbol{W}^{t}-\alpha\nabla\mathcal{F}(\boldsymbol{W}^{t})), \label{3.21}
\end{equation}
where $\Omega=\{\boldsymbol{W}|\mbox{diag}(\boldsymbol{W}\boldsymbol{W}^{H})\leq \boldsymbol{1}\eta\}$ is the domain of $\mathcal{P}2$ and the projection matrix satisfies $\boldsymbol{W}^{t+1}\in \Omega$. Complete pseudo-code is given in Algorithm \ref{algorithm1}.
\begin{lemma}
For any $\boldsymbol{W}^t$, let $\boldsymbol{W}^{t+1}$ be generated by (\ref{3.21}). For any complex matrix $\boldsymbol{W}\in \Omega$, we have 
\begin{equation}
   \mbox{Re}[\boldsymbol{W}^{t+1}-\boldsymbol{W},(\boldsymbol{W}^{t}-\alpha\nabla\mathcal{F}(\boldsymbol{W}^{t}))-\boldsymbol{W}^{t+1}]_F\ge0.\label{3.3}
\end{equation}
\end{lemma}
\begin{proof}
Please refer to Appendix D.
\end{proof}
\begin{algorithm}[t]\label{algorithm1}
\caption{PGD method in the complex domain}
\LinesNumbered
\KwIn{ Step size $\alpha$; tolerant error $\tau$.}
\KwOut{The optimal solution: $\boldsymbol{W}^{*}$.}
Initialize $\boldsymbol{W}^0$\;
\While{$error\ge\tau$}{
$\hat{\boldsymbol{W}^t}=\boldsymbol{W}^t-\alpha\nabla\mathcal{F}(\boldsymbol{W}^t)$\; 
$\boldsymbol{W}^{t+1}=$ \textbf{\texttt{Projection}}$(\hat{\boldsymbol{W}^t},\eta)$ \;
$error=\mathcal{F}(\boldsymbol{W}^t)-\mathcal{F}(\boldsymbol{W}^{t+1})$\;
$\boldsymbol{W}^t=\boldsymbol{W}^{t+1}$\;
} 
 \SetKwFunction{FMain}{\textbf{Projection}}
\SetKwProg{Fn}{function}{:}{}
 \Fn{\FMain{$\boldsymbol{W},\eta$}}{
$\boldsymbol{\ell}=\sqrt{\mbox{diag}(\boldsymbol{W}\boldsymbol{W}^H)}$, 
$\boldsymbol{\iota}=\mbox{find}(\boldsymbol{\ell}>\eta)$\;
$\boldsymbol{W}(\boldsymbol{\iota},:)=\eta\cdot\boldsymbol{W}(\boldsymbol{\iota},:)\oslash[\boldsymbol{\ell}(\boldsymbol{\iota})\boldsymbol{1}^{T}(\boldsymbol{\iota})].$\;}
\textbf{end}\\
\Return { $\boldsymbol{W}^{*}=\boldsymbol{W}^{t}$}\;
 \end{algorithm}

\begin{theorem}
For any $\boldsymbol{W}^t$, let $\boldsymbol{W}^{t+1}$ be generated by the PGD method. If the step size satisfies $\alpha\in(0,1/L)$, the PGD method converges and the following inequalities hold.
\begin{equation}
\begin{aligned}
\mathcal{F}(\boldsymbol{W}^{t})-\mathcal{F}(\boldsymbol{W}^{t+1})\ge (\frac{1}{\alpha}-L)\norm{\boldsymbol{W}^t-\boldsymbol{W}^{t+1}}_F^2,
\end{aligned}\label{yu11}
\end{equation}
\begin{equation}
\begin{aligned}
\norm{\boldsymbol{W}^{t}-\boldsymbol{W}^{opt}}_F^2\ge&\norm{\boldsymbol{W}^{t+1}-\boldsymbol{W}^{opt}}_F^2\\
&+(1-\alpha L)\norm{\boldsymbol{W}^{t}-\boldsymbol{W}^{t+1}}_F^2.
\end{aligned}\label{yu12}
\end{equation}

\label{the3}
\end{theorem}
\begin{proof}
Using the first-order condition of convexity, we have \begin{equation}
\begin{aligned}
\mathcal{F}(\boldsymbol{W}^{t})\ge& \mathcal{F}(\boldsymbol{W}^{t+1})+\mbox{Re}[\boldsymbol{W}^{t}-\boldsymbol{W}^{t+1},\nabla\mathcal{F}(\boldsymbol{W}^{t+1})]_F\\
=& \mathcal{F}(\boldsymbol{W}^{t+1})+\mbox{Re}[\boldsymbol{W}^{t}-\boldsymbol{W}^{t+1},\nabla\mathcal{F}(\boldsymbol{W}^{t})]_F\\
&-\mbox{Re}[\boldsymbol{W}^{t}-\boldsymbol{W}^{t+1},\nabla\mathcal{F}(\boldsymbol{W}^{t})-\nabla\mathcal{F}(\boldsymbol{W}^{t+1})]_F\\
=& \mathcal{F}(\boldsymbol{W}^{t+1})+\mbox{Re}[\boldsymbol{W}^{t}-\boldsymbol{W}^{t+1},\nabla\mathcal{F}(\boldsymbol{W}^{t})]_F\\
&\qquad\qquad\qquad\qquad-\norm{\boldsymbol{H}(\boldsymbol{W}^{t}-\boldsymbol{W}^{t+1})}_F^2\\
\stackrel{(\ref{2.15})}{\ge}& \mathcal{F}(\boldsymbol{W}^{t+1})+\mbox{Re}[\boldsymbol{W}^{t}-\boldsymbol{W}^{t+1},\nabla\mathcal{F}(\boldsymbol{W}^{t})]_F\\
&\qquad\qquad\qquad\qquad-L\norm{\boldsymbol{W}^{t}-\boldsymbol{W}^{t+1}}_F^2\\
\stackrel{(\ref{3.3})}{\ge}& \mathcal{F}(\boldsymbol{W}^{t+1})+\frac{1}{\alpha}\norm{\boldsymbol{W}^{t}-\boldsymbol{W}^{t+1}}_F^2\\
&\qquad\qquad\qquad\qquad-L\norm{\boldsymbol{W}^{t}-\boldsymbol{W}^{t+1}}_F^2\\
=&\mathcal{F}(\boldsymbol{W}^{t+1})+(\frac{1}{\alpha}-L)\norm{\boldsymbol{W}^{t}-\boldsymbol{W}^{t+1}}_F^2.
\end{aligned}
\end{equation}
So the inequality (\ref{yu11}) holds.
From (\ref{us3}) and (\ref{2.15}), we get 
\begin{equation}
\begin{aligned}
\mathcal{F}(\boldsymbol{W})-\mathcal{F}(\boldsymbol{W}^{t+1})\ge
\mbox{Re}[\boldsymbol{W}-\boldsymbol{W}^{t+1},\nabla\mathcal{F}(\boldsymbol{W}^{t})]_F\\
-\frac{L}{2}\norm{\boldsymbol{W}^{t+1}-\boldsymbol{W}^{t}}_F^2.
\end{aligned}\label{us}
\end{equation}
On the other hand, with the projection property, we have 
\begin{equation}
   \mbox{Re}[\boldsymbol{W}-\boldsymbol{W}^{t+1},\nabla\mathcal{F}(\boldsymbol{W}^{t})]_{F}\ge\tfrac{1}{\alpha}\mbox{Re}[\boldsymbol{W}-\boldsymbol{W}^{t+1},\boldsymbol{W}^{t}-\boldsymbol{W}^{t+1}]_F.\label{us15}
\end{equation}
Substituting (\ref{us15}) in (\ref{us}), we obtain
\begin{equation}
\begin{aligned}
\mathcal{F}(\boldsymbol{W})-\mathcal{F}(\boldsymbol{W}^{t+1})\ge&
\frac{1}{\alpha}\mbox{Re}[\boldsymbol{W}-\boldsymbol{W}^{t+1},\boldsymbol{W}^{t}-\boldsymbol{W}^{t+1}]_F\\
&-\frac{L}{2}\norm{\boldsymbol{W}^{t+1}-\boldsymbol{W}^{t}}_F^2.    
\end{aligned}\label{us18}
\end{equation}
Setting $\boldsymbol{W}=\boldsymbol{W}^{opt}$ in (\ref{us18}), we have
\begin{equation}
\begin{aligned}
\mbox{Re}[\boldsymbol{W}^{t}-\boldsymbol{W}^{opt},\boldsymbol{W}^{t}&-\boldsymbol{W}^{t+1}]_F\ge\alpha [\mathcal{F}(\boldsymbol{W}^{t+1})-\mathcal{F}(\boldsymbol{W}^{opt})]\\
&+(1-\alpha L/2)\norm{\boldsymbol{W}^{t+1}-\boldsymbol{W}^{t}}_F^2.
\end{aligned}
\end{equation}
By the above inequality, we find 
\begin{equation}
\begin{aligned}
 &\norm{\boldsymbol{W}^{t+1}-\boldsymbol{W}^{opt}}_F^2=\norm{\boldsymbol{W}^{t}-\boldsymbol{W}^{opt}-(\boldsymbol{W}^{t}-\boldsymbol{W}^{t+1})}_F^2\\
=&\norm{\boldsymbol{W}^{t}-\boldsymbol{W}^{opt}}_F^2+\norm{\boldsymbol{W}^{t}-\boldsymbol{W}^{t+1}}_F^2\\
&\qquad-2\mbox{Re}[\boldsymbol{W}^{t}-\boldsymbol{W}^{opt},\boldsymbol{W}^{t}-\boldsymbol{W}^{t+1}]_{F}\\
\leq&\norm{\boldsymbol{W}^{t}-\boldsymbol{W}^{opt}}_F^2-2\alpha [\mathcal{F}(\boldsymbol{W}^{t+1})-\mathcal{F}(\boldsymbol{W}^{opt})]\\
&\qquad\qquad\qquad\qquad-(1-\alpha L)\norm{\boldsymbol{W}^{t}-\boldsymbol{W}^{t+1}}_F^2.\\
\leq&\norm{\boldsymbol{W}^{t}-\boldsymbol{W}^{opt}}_F^2
-(1-\alpha L)\norm{\boldsymbol{W}^{t}-\boldsymbol{W}^{t+1}}_F^2.\quad
\end{aligned}
\end{equation}
Thus the inequality (\ref{yu12}) holds.
\end{proof}

\section{Numerical experiments}\label{section4}
  In this letter, the dimensions of complex matrices are $M=10$, $N=5$, and $K=8$, respectively. Matrix $\boldsymbol{H}$ and Matrix $\boldsymbol{A}$ are randomly generated in the complex space. The interval of the random number is set to [-10 10]. The value of $\eta$ is set to 2. Our main interest is to show the convergence of GD and PGD. We set different values of the step size to investigate the performance in terms of the convergence rate.

\begin{figure}[t]
    \centering
    \includegraphics[width=3.6in]{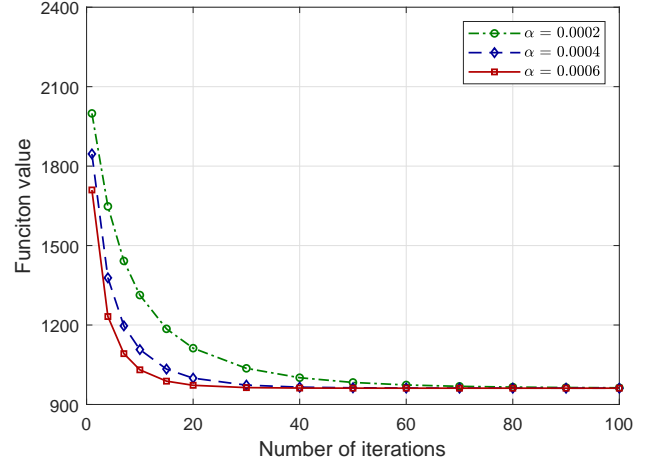}
    \caption{The convergence of the GD method for the unconstrained CMOP.}
    \label{fig:3}
\end{figure}
\begin{figure}[t]
    \centering
    \includegraphics[width=3.6in]{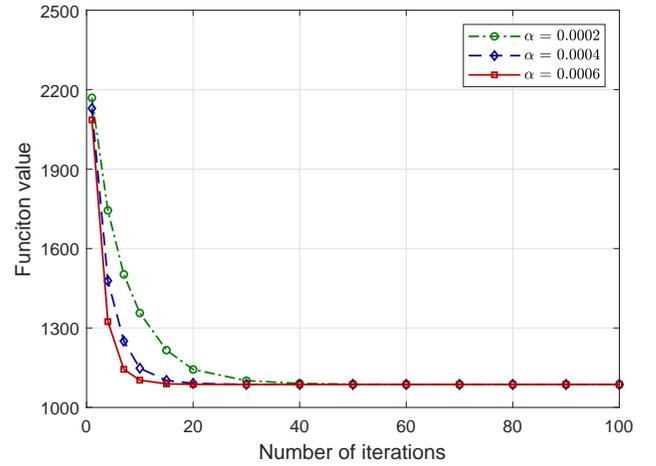}
    \caption{The convergence of the PGD method for the constrained CMOP.}
    \label{fig:5}
\end{figure}
In Fig.\ref{fig:3}, we show the convergence of the GD method for the unconstrained CMOP in Frobenius norm under different values of the step size. We observe that the GD method converges which means that the GD method is feasible for the unconstrained CMOP in Frobenius norm. When the step size is equal to 0.0006 it has a fast convergence rate. When the step size is equal to 0.0002 it has a slow convergence rate. This means that a large step size improves the convergence rate when the GD method converges. Moreover, we also find that the minimizer of the unconstrained CMOP in Frobenius norm is greater than zero.

In Fig. \ref{fig:5}, we demonstrate the convergence of the PGD method for the constrained CMOP in Frobenius norm under different values of the step size. First, the PGD method is feasible for the constrained CMOP in Frobenius norm. Secondly, the minimizer of this problem is also greater than zero. Similar to Fig. \ref{fig:3}, Properly increasing the step size improves the convergence rate. Note that the PGD method and GD method may not converge for given extra large step size. Moreover, we find that the PGD method seems to have a convergence rate similar to that of the GD method.

\section{Conclusion}\label{section5}
Matrix optimization problems in complex variables are frequently encountered in applied mathematics and engineering applications. To cope with them, in this letter, several important optimization theorems, i.e., the first-order condition of convexity, Lipschitz condition, and projection property are firstly extended to the complex domain.  In addition, we proved that iterative methods, i.e., GD and PGD can converge in the complex domain.

\appendices

\section{Proof of Theorem 1}
\begin{proof}
This problem is expanded as 
\begin{equation*}
\begin{aligned}
&\mathcal{F}(\boldsymbol{W})= \frac{1}{2}\norm{\boldsymbol{H}\boldsymbol{W}-\boldsymbol{A}}_{F}^2=\frac{1}{2}\sum_{m=1}^{M}\sum_{k=1}^{K}U_{m,k}^2+V_{m,k}^2,
\end{aligned}
\end{equation*}
where $U_{m,k}=\sum_{n=1}^{N}\bar{h}_{m,n}w_{n,k}^x-\ddot{h}_{m,n}\ddot{w}_{n,k}-\bar{a}_{m,k}$; $V_{m,k}=\sum_{n=1}^{N}\bar{h}_{m,n}\ddot{w}_{n,k}+\ddot{h}_{m,n}\bar{w}_{n,k}-\ddot{a}_{m,k}$. 
Since problem $\mathcal{P}1$ is convex, we optimize the two sets of variables $\{\bar{w}_{n,k},\ddot{w}_{n,k}\}$. In the real domain, if the complex partial derivative satisfies the following condition, we obtain the optimal solution.
\begin{equation}
\begin{aligned}
\frac{\partial\mathcal{F}(\boldsymbol{W})}{\partial \bar{w}_{n,k}}=&\sum_{m=1}^M U_{m,k}\bar{h}_{m,n}+V_{m,k}\ddot{h}_{m,n}=0,\\
\frac{\partial\mathcal{F}(\boldsymbol{W})}{\partial \ddot{w}_{n,k}}=&\sum_{m=1}^M -U_{m,k}\ddot{h}_{m,n}+V_{m,k}\bar{h}_{m,n}=0.\label{2.10}
\end{aligned}
\end{equation}
On the other hand, the complex derivatives of  $\mathcal{F}(\boldsymbol{W})$ is equal to
$\nabla\mathcal{F}(\boldsymbol{W})=\boldsymbol{H}^H(\boldsymbol{H}\boldsymbol{W}-\boldsymbol{A})$ and we have 
\begin{equation}
\begin{aligned}
\{\nabla\mathcal{F}(\boldsymbol{W})\}_{n,k}=&\begin{matrix}\sum_{m=1}^M U_{m,k}\bar{h}_{m,n}+V_{m,k}\ddot{h}_{m,n}\end{matrix}\\
&+(-U_{m,k}\ddot{h}_{m,n}+V_{m,k}\bar{h}_{m,n})i.\label{2.11}
\end{aligned}
\end{equation}
From the formulae (\ref{2.10}) and (\ref{2.11}), we find that the following holds.
\begin{equation}
\begin{aligned}
\frac{\partial\mathcal{F}(\boldsymbol{W})}{\partial \bar{w}_{n,k}}=\mbox{Re}(\{\nabla\mathcal{F}(\boldsymbol{W})\}_{n,k}),\\
\frac{\partial\mathcal{F}(\boldsymbol{W})}{\partial \ddot{w}_{n,k}}=\mbox{Im}(\{\nabla\mathcal{F}(\boldsymbol{W})\}_{n,k}).
\end{aligned}
\end{equation}
For any $n,k$, if formula (\ref{2.10}) holds, we get  $\{\nabla\mathcal{F}(\boldsymbol{W})\}_{n,k}=0$.
Thus the optimal condition is 
$\nabla\mathcal{F}(\boldsymbol{W})=\boldsymbol{H}^H(\boldsymbol{H}\boldsymbol{W}-\boldsymbol{A})=0$ and the optimal solution of the problem $\mathcal{P}1$ is $\boldsymbol{W}^{opt}=(\boldsymbol{H}^H\boldsymbol{H})^{-1}(\boldsymbol{H}^H\boldsymbol{A})$.
\end{proof}

\section{Proof of Lemma 2}
\begin{proof}
Let $\Delta\boldsymbol{W}^t =\boldsymbol{W}^{}-\boldsymbol{W}^t$. 
By Lemma \ref{lemma1}, we have 
\begin{equation*}
\begin{aligned}
    &\lambda\mathcal{F}(\boldsymbol{W}^{})+(1-\lambda)\mathcal{F}(\boldsymbol{W}^t)\ge \mathcal{F}(\lambda\boldsymbol{W}^{}+(1-\lambda)\boldsymbol{W}^t)\\
    &\lambda(\mathcal{F}(\boldsymbol{W}^{})-\mathcal{F}(\boldsymbol{W}^t))\ge \mathcal{F}(\boldsymbol{W}^{t}+\lambda\Delta\boldsymbol{W}^t)-\mathcal{F}(\boldsymbol{W}^t)\\
    &\mathcal{F}(\boldsymbol{W}^{})-\mathcal{F}(\boldsymbol{W}^t)\ge \frac{\mathcal{F}(\boldsymbol{W}^{t}+\lambda\Delta\boldsymbol{W}^t)-\mathcal{F}(\boldsymbol{W}^t)}{\lambda}.
    \end{aligned}
\end{equation*}
    \begin{equation*}
\begin{aligned}
    &\mathcal{F}(\boldsymbol{W}^{})\ge\mathcal{F}(\boldsymbol{W}^t)+\sum_{n=1}^N\sum_{k=1}^K\mathcal{F}_{\bar{w}_{n,k}}(\boldsymbol{W}^t)(\bar{w}_{n,k}^{}-\bar{w}_{n,k}^{t})\\
    &\qquad \qquad  \qquad +\sum_{n=1}^N\sum_{k=1}^K\mathcal{F}_{\ddot{w}_{n,k}}(\boldsymbol{W}^t)(\ddot{w}_{n,k}^{}-\ddot{w}_{n,k}^{t})+o(\lambda)\\
    &\qquad \qquad =\mathcal{F}(\boldsymbol{W}^t)+\mbox{Re}[\boldsymbol{W}^{}-\boldsymbol{W}^{t},\nabla\mathcal{F}(\boldsymbol{W}^{t})]_F+o(\lambda).
\end{aligned}
\end{equation*}
When we take the limit of $\lambda$ to 0, the first-order condition of convexity holds.
\end{proof}
\section{Proof of Lemma 3}
\begin{proof} From the definition of the Lipschitz condition, we have
\begin{equation}
\begin{aligned}
\mathcal{F}(\boldsymbol{W}^{t+1})&\leq \mathcal{F}(\boldsymbol{W}^{t})+\mbox{Re}[\nabla\mathcal{F}(\boldsymbol{W}^{t}),\boldsymbol{W}^{t+1}-\boldsymbol{W}^{t}]_F\\
&\qquad \qquad \qquad \qquad \qquad \quad +\frac{L}{2}\norm{\boldsymbol{W}^{t+1}-\boldsymbol{W}^{t}}_F^2\\
&\leq \mathcal{F}(\boldsymbol{W}^{t})+\mbox{Re}[\boldsymbol{H}\boldsymbol{W}^{t}-\boldsymbol{A},\boldsymbol{H}\boldsymbol{W}^{t+1}-\boldsymbol{H}\boldsymbol{W}^{t}]_F\\
&\qquad \qquad \qquad \qquad \qquad \quad +\frac{L}{2}\norm{\boldsymbol{W}^{t+1}-\boldsymbol{W}^{t}}_F^2\\
&\leq -\mathcal{F}(\boldsymbol{W}^{t})+\mbox{Re}[\boldsymbol{H}\boldsymbol{W}^{t}-\boldsymbol{A},\boldsymbol{H}\boldsymbol{W}^{t+1}-\boldsymbol{A}]_F\\
&\qquad \qquad \qquad \qquad \qquad \quad +\frac{L}{2}\norm{\boldsymbol{W}^{t+1}-\boldsymbol{W}^{t}}_F^2.
\end{aligned}
\end{equation}
From the above formula, we obtain
\begin{equation}
\begin{aligned}
\norm{\boldsymbol{H}(\boldsymbol{W}^{t+1}-\boldsymbol{W}^{t})}_F^2\leq L\norm{\boldsymbol{W}^{t+1}-\boldsymbol{W}^{t}}_F^2.
\end{aligned}
\end{equation}
Thus Lipschitz constant $L$ of the unconstrained CMOP depends on the complex matrix $\boldsymbol{H}$.
\end{proof}

\section{Proof of Lemma 4}
\begin{proof} In the real domain, the constraint is convex and it can be written as  
\begin{equation}
    \sum_{k=1}^K \bar{w}_{n,k}^2+\ddot{w}_{n,k}^2\leq \eta, \forall n,
\end{equation}
\begin{equation}
    \sum_{n=1}^N\sum_{k=1}^K \bar{w}_{n,k}^2+\ddot{w}_{n,k}^2\leq N\eta.
\end{equation}
 By the projection's property, we have 
\begin{equation}
    (\boldsymbol{u}-\boldsymbol{u}^{t+1})^T(\boldsymbol{u}^t-\alpha\boldsymbol{v}^t-\boldsymbol{u}^{t+1})\leq0, \forall \boldsymbol{u}\in \Omega.\label{3.5}
\end{equation}
where $\boldsymbol{v}=[vec(\mbox{Re}(\nabla\mathcal{F}(\boldsymbol{W})))\ vec(\mbox{Im}(\nabla\mathcal{F}(\boldsymbol{W})))]$, $\boldsymbol{u}=[vec(\mbox{Re}(\boldsymbol{W}))\ vec(\mbox{Im}(\boldsymbol{W}))]$ and $\boldsymbol{u}^{t+1}=P_{\Omega}(\boldsymbol{u}^t-\alpha\boldsymbol{v}^t)$. The inequality (\ref{3.5}) is equivalent to the following state.
For any matrix $ \boldsymbol{W} \in \Omega$, we have
\begin{equation}
\mbox{Re}[\boldsymbol{W}^{t+1}-\boldsymbol{W},(\boldsymbol{W}^{t}-\alpha\nabla\mathcal{F}(\boldsymbol{W}^{t}))-\boldsymbol{W}^{t+1}]_F\ge0.
\end{equation}
Thus this lemma holds.
\end{proof}


\bibliographystyle{ieeetr}
\bibliography{ref}

\end{document}